\documentclass[reqno,12pt,a4paper]{amsart}
\usepackage[T1]{fontenc}
\usepackage[utf8]{inputenc}
\usepackage[english]{babel}
\usepackage{amsmath,amssymb,amsthm}
\usepackage{bm}
\usepackage[margin=25mm]{geometry}
\usepackage[hidelinks]{hyperref}
\newcommand{\cN}{\mathbb N}
\newcommand{\cZ}{\mathbb Z}
\newcommand{\PP}{\mathbb P}
\newcommand{\cA}{\mathcal A}
\newcommand{\cB}{\mathcal B}
\newcommand{\cS}{\mathcal S}
\newcommand{\bu}{\boldsymbol u}
\newcommand{\bv}{\boldsymbol v}

\theoremstyle{plain}
\newtheorem{theorem}{Theorem}[section]
\newtheorem{lemma}{Lemma}[section]

\theoremstyle{definition}
\newtheorem{example}{Example}[section]
\theoremstyle{remark}
\newtheorem{remark}{Remark}
\numberwithin{equation}{section}

\begin{document}

\title[A Romanoff-type theorem]{A Romanoff-type theorem for a multiset of products of powers}
\author{Artyom Radomskii}
\date{}
\begin{abstract}
Let $b_1,\dots,b_d\ge2$ be fixed integers, and let $k$ be their
multiplicative rank. We study representations
$n=a+b_1^{u_1}\cdots b_d^{u_d}$, where $a$ belongs to a set
$\cA$ of positive integers and $u_1,\dots,u_d$ are positive integers,
counting distinct tuples of exponents separately. Under density and correlation assumptions
on $\cA$, we obtain a lower bound for the number of integers with many
such representations, in terms of $d$ and $k$. In particular, when
$\cA$ is the set of primes or the set of positive integers representable
as a sum of two squares, a positive proportion of the integers $n\le x$
have at least $c_1(\log x)^{d-1}$ or $c_1(\log x)^{d-1/2}$
representations, respectively, for some $c_1>0$ and all sufficiently
large $x$. No multiplicative independence or coprimality assumptions on
the bases are required.
\end{abstract}
\address{HSE University, Moscow, Russian Federation}


\email{artyom.radomskii@mail.ru}

\maketitle

\section{Introduction}

N.~P. Romanoff \cite{Romanoff} proved that the set of positive integers representable as the sum of a prime number and a power of a given integer $a>1$ has a positive lower asymptotic density. In this paper we obtain a generalization of his result.

Let $d\ge1$ and let $b_1,\dots,b_d\ge2$ be arbitrary fixed integers.
For $\bu=(u_1,\dots,u_d)\in\cN^d$, where $\cN=\{1,2,\dots\}$, put
\[
  \beta(\bu)=b_1^{u_1}\cdots b_d^{u_d}.
\]
We regard the family
\[
  \cB=\bigl(\beta(\bu)\bigr)_{\bu\in\cN^d}
\]
as a multiset: each tuple of exponents contributes a separate occurrence,
even when different tuples yield the same product.

For a set $\cA\subseteq\cN$, define
\[
  R_{\cA}(n)=
  \#\{(a,\bu)\in\cA\times\cN^d:a+\beta(\bu)=n\}.
\]

Let $k$ denote the multiplicative rank of the bases, that is, the maximum
number of multiplicatively independent elements among $b_1,\dots,b_d$.
Recall that $a_1,\dots,a_m$ are multiplicatively independent if
\[
  a_1^{t_1}\cdots a_m^{t_m}=1,\quad t_1,\dots,t_m\in\cZ
  \quad\Longrightarrow\quad t_1=\cdots=t_m=0.
\]
We have $1\le k\le d$.

Our main result is the following theorem.

\begin{theorem}\label{T1}
Let $\cA\subseteq\cN$ be an infinite set, and let $\eta$ be a positive function defined
for all sufficiently large real $x$. Put
\[
  \cA(x)=\#\{a\in\cA:a\le x\},\qquad
  \cA(x,h)=\#\{a\in\cA:a\le x,\ a+h\in\cA\}.
\]
Suppose that, for all sufficiently large $x$,
\begin{equation}\label{eq:A-count}
  \cA(x)\asymp\frac{x}{\eta(x)},\qquad
  \cA(x/2)\gg\cA(x)
\end{equation}
and
\begin{equation}\label{eq:A-correlation}
  \cA(x,h)\ll\frac{h}{\varphi(h)}\frac{x}{\eta(x)^2}
\end{equation}
uniformly in integers $h\ge1$. Then there exist positive constants
$c_1,c_2$, depending only on $b_1,\dots,b_d$ and the constants in
\eqref{eq:A-count} and \eqref{eq:A-correlation}, such that
\begin{equation}\label{eq:main}
  \#\left\{n\le x:
    R_{\cA}(n)\ge c_1\frac{(\log x)^d}{\eta(x)}\right\}
  \ge c_2x\,\frac{(\log x)^k}{(\log x)^k+\eta(x)}
\end{equation}
for all sufficiently large $x$.
\end{theorem}

\begin{remark}The threshold for the number of representations in \eqref{eq:main}
involves the number of bases $d$, whereas the right-hand side involves
their multiplicative rank $k$. Neither multiplicative independence nor
pairwise coprimality of the bases is assumed in the theorem.
\end{remark}

\begin{remark}
If the bases are multiplicatively independent, then $k=d$ and every
product value has multiplicity $1$. In this case, counting with the set
and with the multiset gives the same result. In particular, this holds
if the numbers $b_1,\ldots,b_d$ are pairwise coprime.
\end{remark}

Let $\PP$ denote the set of primes, and let
$\cS=\{u^2+v^2:u,v\in\cZ\}\cap\cN$. From Theorem \ref{T1} we obtain the following result.

\begin{theorem}\label{T2}
For arbitrary fixed integer bases $b_1,\dots,b_d\ge2$, put
\begin{align*}
  R_1(n)&=\#\{(p,\bu)\in\PP\times\cN^d:p+\beta(\bu)=n\},\\
  R_2(n)&=\#\{(s,\bu)\in\cS\times\cN^d:s+\beta(\bu)=n\}.
\end{align*}
Then there exist positive constants $c_1,c_2$, depending only on the
bases, such that
\begin{equation}\label{eq:prime-corollary}
  \#\{n\le x:R_1(n)\ge c_1(\log x)^{d-1}\}\ge c_2x
\end{equation}
and
\begin{equation}\label{eq:square-corollary}
  \#\{n\le x:R_2(n)\ge c_1(\log x)^{d-1/2}\}\ge c_2x
\end{equation}
for all sufficiently large $x$.
\end{theorem} The first assertion of Theorem \ref{T2} with $d=2$, $b_1 = 2$, and $b_2=3$ answers a question posed by Kevin Ford in private correspondence.

\begin{example}\label{ex:repeated-base}
If $b_1=b_2=2$, then $d=2$, $k=1$, and the number $2^m$, $m\ge2$,
occurs with multiplicity $m-1$. Hence
\[
  R_1(n)=\sum_{\substack{m\ge2\\n-2^m\in\PP}}(m-1).
\]
Theorem~\ref{T2} states that
\[
  \#\{n\le x:R_1(n)\ge c_1\log x\}\ge c_2x
\]
for all sufficiently large $x$. 
\end{example}

\section{Notation}

We use $X\ll Y$, $Y\gg X$, or $X=O(Y)$ to denote the estimate $|X|\leq C Y$ for some constant $C>0$, and write $X\asymp Y$ for $X \ll Y \ll X$. The notation $X\sim Y$  means that $\lim_{x\to \infty} X/Y = 1$.

We write $\cZ$ for the set of integers, $\cN$ for the set of positive
integers, $\PP$ for the set of primes, $\cS$ for the set of positive
integers representable as $u^2+v^2$ with $u,v\in\cZ$, and $\cS'$
for the set of odd positive integers representable as $u^2+v^2$ with
$u,v\in\cZ$ and $\gcd(u,v)=1$.

 We reserve the letter $p$ for primes. In particular, the sum $\sum_{p\leq K}$ should be interpreted as being over all prime numbers not exceeding $K$. We write $\varphi (n)$ for Euler's totient function (the order of the multiplicative group of reduced residue classes modulo $n$).

If $x$ is a real number, then $\lfloor x\rfloor$ denotes its integral part. By $\# A$ we denote the number of elements of a finite set $A$. All logarithms are natural.

\section{Proofs of Theorems \ref{T1} and \ref{T2}}

We need the following result.
\begin{lemma}\label{L1}
Let $\cA\subseteq\cN$ be an infinite set, and let
$\mathcal D=(d_n)_{n=1}^{\infty}$ be a sequence of positive integers
(not necessarily distinct). Let $x_0\ge2$, and let
$\eta:[x_0,\infty)\to(0,\infty)$ be a function. Suppose that
\[
\textup{ord}_{\mathcal{D}}(v)=\#\{i\in \mathbb{N}: d_i=v\}<\infty
\]for all $v\in \mathbb{N}$. Put
\[\mathcal{D}(x)=\#\{n\in \cN: d_n\le x\}.
\]
Assume that \eqref{eq:A-count} and \eqref{eq:A-correlation} hold for
every real $x\ge x_0$, with \eqref{eq:A-correlation} uniform in integers
$h\ge1$. Suppose also that, for every real $x\ge x_0$,
\begin{equation}
\mathcal{D}(x/2)\gg \mathcal{D}(x)\label{T1:Basic.2}
\end{equation}and
\begin{equation}\label{T1:Basic.3}
\sum_{\substack{j\in \mathbb{N}:\\ d_{j}\leq x}} \sum_{p\leq (\log x)^{\alpha}}\frac{\lambda_{\mathcal{D}} (x;j,p)\log p}{p} \ll \mathcal{D}(x)^{2}
\end{equation}for some $\alpha\in (0,1)$, where
\[\lambda_{\mathcal{D}}(x;j, p) = \#\{k\in \mathbb{N}: d_k \leq x\text{ and } d_k\equiv d_j\textup{ (mod $p$)}\}.
\]Set
\[
R(n)=\#\{(a,i)\in \cA\times\cN: a+d_i = n\}\quad\text{and}\quad \rho_{\mathcal{D}}(x)=\max_{1\leq v\leq x} \textup{ord}_{\mathcal{D}}(v).
\]Then there exist positive constants $\gamma_1$ and $\gamma_2$ depending only on $\alpha$ and the constants implied by the symbols $\ll$, $\gg$, and $\asymp$ in \eqref{eq:A-count}, \eqref{eq:A-correlation}, \eqref{T1:Basic.2}, and \eqref{T1:Basic.3} such that
\[
\#\Big\{n\leq x: R(n)\geq \gamma_1 \frac{\mathcal{D}(x)}{\eta(x)}\Big\}\geq \gamma_2 x \frac{\mathcal{D}(x)}{\mathcal{D}(x)+ \rho_{\mathcal{D}} (x)\eta(x)}
\]for all $x \ge x_0$.
\end{lemma}
\begin{proof}
This is Theorem~1.1 in \cite{Radomskii}; see also Theorem~1.6 in \cite{Radomskii.Izv} for the case $\cA=\PP$.
\end{proof}

\begin{proof}[Proof of Theorem~\ref{T1}]
We apply Lemma \ref{L1}, which allows repetitions in the
sequence $\cB$ and takes their maximum multiplicity into account.
The tuples in $\cN^d$ can be enumerated to obtain a sequence with all the
prescribed repetitions. All the counting functions defined below are
finite, since every base is greater than $1$.

Put $L=\log x$. The counting function, with multiplicities included, is
\[
  \cB(x)=\#\{\bu\in\cN^d:\beta(\bu)\le x\}
  =\#\left\{\bu\in\cN^d:
       \sum_{i=1}^d u_i\log b_i\le L\right\}.
\]
Let us show that
\begin{equation}\label{eq:B-count}
  \cB(x)=\frac{L^d}{d!\prod_{i=1}^d\log b_i}
  +O_{b_1,\dots,b_d}(L^{d-1}).
\end{equation} Put
\[
\ell_i=\log b_i>0,\qquad
H=\ell_1+\cdots+\ell_d.
\]
Then
\[
\cB(x)=\#\left\{\bm{u}\in\cN^d:
                    \sum_{i=1}^d\ell_i u_i\le L\right\}.
\]
For $T\ge0$, consider the simplex
\[
\Delta_T=\left\{\bm{t}\in[0,\infty)^d:
                    \sum_{i=1}^d\ell_i t_i\le T\right\}.
\]Recall that the standard simplex
$\{\bm{\xi}\in[0,\infty)^d:
\sum_{i=1}^d\xi_i\le1\}$ has volume $1/d!$.
Hence, by scaling and the change of variables
$y_i=\ell_i t_i$, we obtain
\[
\operatorname{vol}(\Delta_T)
=\frac{T^d}{d!\,\ell_1\cdots\ell_d}.
\]
To each tuple $\bm{u}\in\cN^d$ counted by $\cB(x)$, associate the unit cube
\[
Q_{\bm{u}}=\prod_{i=1}^d[u_i-1,u_i).
\]
These cubes are pairwise disjoint, so the volume of their union equals
$\cB(x)$. Moreover, for $L\ge H$,
\[
\Delta_{L-H}
\subseteq
\bigcup_{\substack{\bm{u}\in\cN^d\\\sum_{i=1}^d\ell_i u_i\le L}}
Q_{\bm{u}}
\subseteq\Delta_L.
\]
The right inclusion follows from $t_i<u_i$ whenever
$\bm{t}\in Q_{\bm{u}}$. To prove the left inclusion, take
$\bm{t}\in\Delta_{L-H}$ and put $u_i=\lfloor t_i\rfloor+1$.
Then $\bm{t}\in Q_{\bm{u}}$ and
\[
\sum_{i=1}^d\ell_i u_i
\le\sum_{i=1}^d\ell_i t_i+H
\le L.
\]
Comparing volumes, we obtain
\[
\frac{(L-H)^d}{d!\,\ell_1\cdots\ell_d}
\le\cB(x)
\le\frac{L^d}{d!\,\ell_1\cdots\ell_d}.
\]
Since $H$ depends only on the fixed bases,
\[
L^d-(L-H)^d=O_{b_1,\dots,b_d}(L^{d-1}).
\]
The two bounds therefore imply \eqref{eq:B-count}.

The asymptotic formula \eqref{eq:B-count} yields
\begin{equation}\label{eq:B-doubling}
  \cB(x)\asymp L^d,\qquad \cB(x/2)\gg \cB(x).
\end{equation} Thus, \eqref{T1:Basic.2} holds.

We estimate the maximum multiplicity
\[
  \rho_{\cB}(x)=\max_{1\le m\le x}
  \#\{\bu\in\cN^d:\beta(\bu)=m\}.
\]
After relabelling the bases, we may assume that $b_1,\dots,b_k$ are
multiplicatively independent. For fixed $u_{k+1},\dots,u_d$, the equation
\[
  b_1^{u_1}\cdots b_k^{u_k}
  =\frac{m}{b_{k+1}^{u_{k+1}}\cdots b_d^{u_d}}
\]
has at most one solution in $u_1,\dots,u_k$.
If $m\le x$, each of the remaining exponents is at most
$L/\log b_i$. Hence
\begin{equation}\label{eq:rho}
  \rho_{\cB}(x)\ll L^{d-k}.
\end{equation}
If $k=d$, there are no free exponents, and every value has multiplicity
$1$.

We now verify the condition on the distribution in residue classes.
For a tuple $\bu$ with $\beta(\bu)\le x$, put
\[
  \lambda(x;\bu,p)=
  \#\{\bv\in\cN^d:\beta(\bv)\le x,
     \ \beta(\bv)\equiv\beta(\bu)\pmod p\}.
\]
Write $D=b_1\cdots b_d$. For primes $p\mid D$, we use the bound
\begin{equation}\label{eq:bad-primes}
  \lambda(x;\bu,p)\le \cB(x).
\end{equation}

For a prime $p\nmid b_1$, let $h_p$ denote the multiplicative
order of $b_1$ modulo $p$, which is to say that $h_{p}$ is the least positive integer $h$ such that $b_1^{h}\equiv 1$ (mod $p$). By Fermat's theorem, $b_1^{p-1}\equiv 1$ (mod $p$), and hence $h_{p}$ exists and $1\leq h_{p} \leq p-1$.

Now suppose that $p\nmid D$ and $p\le L^{1/2}$.
For fixed $v_2,\dots,v_d$, the congruence
\[
  b_1^{v_1}\cdots b_d^{v_d}\equiv\beta(\bu)\pmod p
\]
either has no solutions or determines a single residue class for $v_1$
modulo $h_p$. Each exponent $v_i$ is bounded above by $L/\log b_i$.
Consequently,
\[
  \lambda(x;\bu,p)
  \ll L^{d-1}\left(\frac{L}{h_p}+1\right).
\]
Since
$h_p\le p-1\le L^{1/2}$, we obtain
\begin{equation}\label{eq:lambda}
  \lambda(x;\bu,p)\ll\frac{L^d}{h_p}
  \ll\frac{\cB(x)}{h_p}.
\end{equation}
All these estimates are uniform in the tuples $\bu$ and primes $p$
under consideration.

By \cite[Lemma~4.2]{Radomskii}, applied with base $b_1$ and
$\varepsilon=1$,
\begin{equation}\label{eq:order-series}
  \sum_{p\nmid b_1}\frac{\log p}{p\,h_p}<\infty.
\end{equation}
It follows from \eqref{eq:bad-primes}, \eqref{eq:lambda}, and
\eqref{eq:order-series} that
\begin{align*}
  \sum_{p\le L^{1/2}}\frac{\lambda(x;\bu,p)\log p}{p}
  &\ll \cB(x)\left(
      \sum_{p\mid D}\frac{\log p}{p}
      +\sum_{p\nmid D}\frac{\log p}{p\,h_p}\right)\\
  &\ll \cB(x).
\end{align*}
Summing over all tuples $\bu$ with $\beta(\bu)\le x$, we obtain
\begin{equation}\label{eq:residue-condition}
  \sum_{\substack{\bu\in\cN^d\\\beta(\bu)\le x}}
  \sum_{p\le(\log x)^{1/2}}
  \frac{\lambda(x;\bu,p)\log p}{p}
  \ll \cB(x)^2.
\end{equation}

In view of \eqref{eq:B-doubling} and
\eqref{eq:residue-condition}, all the hypotheses
of Lemma~\ref{L1} hold with $\alpha=1/2$. This lemma yields positive constants
$\gamma_1,\gamma_2$ such that
\begin{equation}\label{eq:criterion}
  \#\left\{n\le x:
    R_{\cA}(n)\ge\gamma_1\frac{\cB(x)}{\eta(x)}\right\}
  \ge\gamma_2x\,\frac{\cB(x)}{\cB(x)+\rho_{\cB}(x)\eta(x)}.
\end{equation}
By \eqref{eq:B-count} and \eqref{eq:rho},
\[
  \frac{\cB(x)}{\cB(x)+\rho_{\cB}(x)\eta(x)}
  \gg\frac{L^d}{L^d+L^{d-k}\eta(x)}
  =\frac{L^k}{L^k+\eta(x)}.
\]
Together with $\cB(x)\asymp L^d$ and \eqref{eq:criterion}, this proves
\eqref{eq:main} after adjusting the positive constants. Theorem \ref{T1} is proved.
\end{proof}

\begin{proof}[Proof of Theorem~\ref{T2}]
For $\cA=\PP$, conditions \eqref{eq:A-count} and
\eqref{eq:A-correlation} hold with $\eta(x)=\log x$, by the prime number
theorem and the standard sieve upper bound for pairs of primes
\cite[Corollary~2.4.1]{HalberstamRichert}.
Since $k\ge1$, for large $x$ we have
\[
  \frac{(\log x)^k}{(\log x)^k+\log x}\ge\frac12.
\]
Applying Theorem~\ref{T1} gives \eqref{eq:prime-corollary}.

For the second assertion, consider the set $\cS'$ of odd positive integers
representable as $u^2+v^2$ with $u,v\in\cZ$ and $\gcd(u,v)=1$.
This is precisely the set of positive integers all of whose prime
divisors are congruent to $1$ modulo $4$, including the integer $1$.

The standard asymptotic formula
(see \cite[Theorem~14.2]{Friedlander.Iwaniec}, with the leading
constant corrected) is
\[
\mathcal{S}'(x)\sim  \frac{cx}{\sqrt{\log x}},
\]where
\[
c=\frac{1}{2\sqrt{2}}\prod_{p \equiv 3\textup{\,(mod $4$)}} \Big(1 - \frac{1}{p^{2}}\Big)^{1/2}.
\]
The value of $c$ can be checked using the corresponding Euler
product.\footnote{For real $s>1$, put
\[
F(s)=\sum_{n\in\cS'}n^{-s}
    =\prod_{p\equiv1\pmod4}(1-p^{-s})^{-1}.
\]
If $\chi_4$ is the nonprincipal Dirichlet character modulo $4$, then
\[
F(s)^2=(1-2^{-s})\zeta(s)L(s,\chi_4)
       \prod_{p\equiv3\pmod4}(1-p^{-2s}).
\]
Since $L(1,\chi_4)=\pi/4$ and $\Gamma(1/2)=\sqrt\pi$, the leading
coefficient in the asymptotic formula is
\[
c=\frac{1}{\sqrt\pi}\lim_{s\to1^+}(s-1)^{1/2}F(s)
 =\frac{1}{2\sqrt2}
   \prod_{p\equiv3\pmod4}(1-p^{-2})^{1/2}.
\]
}
Moreover, by \cite[Corollary~2.3.4]{HalberstamRichert}, uniformly
in integers $h\ge1$,
\[
\mathcal{S}' (x, h) \ll \prod_{\substack{p|h\\ p \equiv 3 \textup{\,(mod $4$)}}} \Big(1 - \frac{1}{p}\Big)^{-1} \frac{x}{\log x} \leq \frac{h}{\varphi (h)} \frac{x}{\log x},
\]where the constant implied by the $\ll$-symbol is absolute. We see that the set $\mathcal{A}=\mathcal{S}'$ satisfies \eqref{eq:A-count} and \eqref{eq:A-correlation} with $\eta(x)=(\log x)^{1/2}$.

We put
\[
R_2^{\prime}(n)=\#\{(s,\bu)\in\cS^{\prime}\times\cN^d:s+\beta(\bu)=n\}.
\]Apply Theorem~\ref{T1} with $\cA=\cS'$ and
$\eta(x)=\sqrt{\log x}$. This theorem yields positive constants
$c_1,c_2$, depending only on the bases $b_1,\ldots,b_d$, such that
\begin{equation}\label{FINAL}
  \#\left\{n\le x:
     R_2^{\prime}(n)\ge c_1\,(\log x)^{d-1/2}\right\}
  \ge c_2 x\,\frac{(\log x)^k}{(\log x)^k+\sqrt{\log x}}
\end{equation}
for all sufficiently large $x$. Since $k\ge1$,
\[
  \frac{(\log x)^k}{(\log x)^k+\sqrt{\log x}}\ge\frac12
\]
for large $x$. Since $\cS^{\prime}\subseteq \cS$, we have
$R_2^{\prime}(n)\le R_2(n)$ for all positive integers $n$.
Therefore \eqref{FINAL} yields \eqref{eq:square-corollary},
after adjusting the positive constants. This completes the proof of Theorem \ref{T2}.
\end{proof}

\section{Acknowledgements}

This work is an output of a research project (HSE-BR-2025-024) implemented as part of the Basic Research Program at HSE University.

\end{document}